\documentclass{amsart}

\usepackage{amsmath, amsfonts, mathrsfs}
\usepackage{amssymb,latexsym}
\usepackage{enumerate, bbm}

\newtheorem{theorem}{Theorem}[section]

\newtheorem{lemma}[theorem]{Lemma}

\newtheorem{example}[theorem]{Example}

\theoremstyle{definition}
\newtheorem{definition}[theorem]{Definition}

\numberwithin{equation}{section}

\begin{document}

\title{Descriptive characterizations of the integral by seminorms}

\author{ Sokol Kaliaj, Zenepe Shkoza}

\address{University of Elbasan,
Mathematics Department,
Elbasan, Albania}

\email{sokol\_bush@yahoo.co.uk}

\address{University of Elbasan,
Mathematics Department,
Elbasan, Albania}

\email{shnepi2014@gmail.com}

\thanks{}



\begin{abstract}
In this paper, 
we first define the concept of the limit average range of a function defined on $[0,1]$ and taking values in a 
Hausdorff locally convex topological vector space (locally convex space) $X$. 
Then, we present characterizations of the primitive $F:[0,1] \to X$ of an integrable by seminorms 
function $f:[0,1] \to X$ 
in terms of the limit average range of $F$.
It is shown that the limit average range characterizes the integral by seminorms 
better then the usual differential.
\end{abstract}

\maketitle

\section{ Introduction and Preliminaries }

The integral by seminorms is an extension of Bochner integral to locally convex spaces. 
The following theorem gives necessary and sufficient conditions 
for a Banach space valued function $F:[0,1] \to X$ to be the primitive of a 
Bochner integrable function $f:[0,1] \to X$  
in terms of the derivative of $F$, c.f. Theorem 7.4.15 in \cite{GUS}. 
\begin{theorem}\label{T1.1}
Let $X$ be a Banach space and let $f,F : [0; 1] \to X$ be functions. 
Then, the following are equivalent:
\begin{itemize}
\item[(i)]
$f$ is Bochner integrable on $[0,1]$ with the primitive $F$, i.e., 
$$
\widetilde{F}(I) = (B)\int_{I} f(t)d\lambda(t),
\text{ for all }
I \in \mathcal{I},
$$ 
where $\mathcal{I}$ is the family of all non-degenerate closed subintervals of $[0,1]$,  
$\lambda$ is the Lebesgue measure on $[0,1]$ and 
$
\widetilde{F}([u,v]) = F(v)-F(u),
$
for all $[u,v] \in \mathcal{I}$,  
\item[(ii)]
$F$ is strongly absolutely continuous, 
$F'(t)$ exists and $F'(t)=f(t)$, at almost all $t \in [0,1]$.
\end{itemize}
\end{theorem}

In this paper, 
functions defined on $[0,1]$ and taking values in a 
locally convex space $X$ are considered. 
At first, the concept of the limit average range of a function $F:[0,1] \to X$ at a point $t \in [0,1]$ 
is defined. 
Then, we present characterizations of the primitive $F:[0,1] \to X$ of 
an integrable by seminorms function $f:[0,1] \to X$ in terms of the limit average range of $F$, 
Theorem \ref{T2.1} and Theorem \ref{T2.2}. 
We give an example of a function $F:[0,1] \to X$ which is the primitive of an integrable by seminorms function $f:[0,1] \to X$ 
and it is shown that $F$ has the limit average range at all $t \in (0,1)$, 
but $F$ is not differentiable at any $t \in (0,1)$, 
Example \ref{Example2.1}.

Let $X$ be a locally convex space with the topology $\tau$ 
and let $\mathscr{P}$ be the family of all continuous seminorms on $X$.  
For any $p \in \mathscr{P}$,   
we denote by $\widetilde{X}_{p}$ the quotient vector space $X/p^{-1}(0)$, by $\phi_{p}:X \to \widetilde{X}_{p}$
the canonical quotient map, 
by $(\widetilde{X}_{p},\widetilde{p})$ the quotient normed space  
and by $(\overline{X}_{p},\overline{p})$ the completion of $(\widetilde{X}_{p},\widetilde{p})$. 
For every $p,q \in \mathscr{P}$ such that $p \leq q$, we denote by $\widetilde{g}_{pq} : 
\widetilde{X}_{q} \rightarrow \widetilde{X}_{p}$  the map defined as follows
$$
\widetilde{g}_{pq}(w_{q})=w_{p},
\text{ for each }
w_{q} \in \widetilde{X}_{q}, 
$$
where $w_{p} = \phi_{p}(x)$, for some vector $x \in w_{q}$. 
By  
$
\overline{g}_{pq} 
$
the continuous linear extension of $\widetilde{g}_{pq}$ to $\overline{X}_{q}$ is denoted.

We now denote by 
$$
\underset{\longleftarrow}{\lim} ~\overline{g}_{pq}~\overline{X}_{q}
\qquad
(~\underset{\longleftarrow}{\lim} ~\widetilde{g}_{pq}~\widetilde{X}_{q}~)
$$ 
\textit{the projective limit} of the family 
$$
\{(\overline{X}_{p}, \overline{p}): p \in \mathscr{P}\}
\qquad
(~\{(\widetilde{X}_{p}, \widetilde{p}): p \in \mathscr{P}\}~)
$$ 
with respect to the family 
$$
\{ \overline{g}_{pq}: p,q \in \mathscr{P}, p \leq q \}
\qquad
(~\{ \widetilde{g}_{pq}: p,q \in \mathscr{P}, p \leq q \}~), 
$$
c.f \cite{SCH}, p.52. 
By virtue of II.5.4 in \cite{SCH}, p.53, 
if $X$ is a complete locally convex space, 
then we have
\begin{equation}\label{eqSCH}
X 
\equiv 
\underset{\longleftarrow}{\lim} ~\widetilde{g}_{pq}~\widetilde{X}_{q} 
\qquad
\text{and}
\qquad
\underset{\longleftarrow}{\lim} ~\widetilde{g}_{pq}~\widetilde{X}_{q}
=
\underset{\longleftarrow}{\lim} ~\overline{g}_{pq}~\overline{X}_{q},
\end{equation} 
where the symbol '$\equiv$' shows that the two spaces above are isomorphic.


Assume that a function $F:[0,1] \to X$ and a point $t \in [0,1]$ are given.  
We put
\begin{equation*}
\Delta F(t,h) = \frac{F(t+h)-F(t)}{h},~(~h \neq 0~)
\qquad
A_{F}(t,\delta) = 
\{ \Delta F(t,h) : 0< |h| < \delta  \}
\end{equation*}
and  
\begin{equation*}
A_{F}^{(p)}(t)= 
\bigcap_{\delta > 0} \overline{A}_{F}^{(p)}(t,\delta)
\qquad
A_{F}(t) = \bigcap_{p \in \mathscr{P}} A_{F}^{(p)}(t)
\end{equation*}
where $\overline{A}_{F}^{(p)}(t,\delta)$ is the closure of $A_{F}(t,\delta)$ with respect to 
the seminorm $p \in \mathscr{P}$. 
The set $A_{F}(t)$ is said to be \textit{the average range} of $F$ at $t$. 
Since 
$$
\overline{A}_{F}(t,\delta) = 
\bigcap_{p \in \mathscr{P}} \overline{A}_{F}^{(p)}(t,\delta)
$$
we have also
$$
A_{F}(t) = \bigcap_{\delta >0} \overline{A}_{F}(t,\delta).
$$
Define
$$
\textit{p-diam}(W) = \sup \{ p(x-y) : x,y \in W \}
\qquad
(~W \subset X, W \neq \emptyset~),
$$
and $\textit{p-diam}(\emptyset) = 0$. 
We say that $F$ has \textit{the limit average range} at $t \in [0,1]$, 
if $A_{F}(t)$ is a bounded set, 
and for each $\varepsilon >0$ and $p \in \mathscr{P}$ 
there exists $\delta_{\varepsilon,p} >0$ such that 
$$
\textit{p-diam} 
\left ( 
A_{F}(t,\delta_{\varepsilon,p}) 
\right )
<
\textit{p-diam} 
\left ( 
A_{F}(t)
\right ) 
+ \varepsilon.
$$  
The  function $F$ is said to be 
\textit{differentiable} at the point $t$, 
if there exists a vector  
$x \in X$ such that for each $p \in \mathscr{P}$, we have
\begin{equation*}
\lim_{h \to 0} 
p(\Delta F(t,h) - x)=0.
\end{equation*}
By $x = F'(t)$ the derivative of $F$ at $t$ is denoted.

We now recall the concept of the integral by seminorms, 
Definition 2.4 in \cite{BLOND}.

\begin{definition}\label{D1.1}
A function $f:[0,1] \to X$ is said to be 
\textit{integrable by seminorms} if, for any 
$p \in \mathscr{P}$, there exists a sequence 
$(f_{k}^{(p)})_{k}$ of measurable simple functions and a subset 
$Z_{p} \subset [0,1]$ with $\lambda(Z_{p})=0$, such that
\begin{itemize}
\item[(i)]
for all $t \in [0,1] \setminus Z_{p}$, we have
$$
\lim_{k \to \infty} p \left ( f_{k}^{(p)}(t)-f(t) \right ) =0,
$$
\item[(ii)]
each function $p \circ (f_{k}^{(p)}-f)$ is Lebesgue integrable on $[0,1]$ and   
$$
\lim_{k \to \infty} \int_{[0,1]} p 
\left ( f_{k}^{(p)}(t)-f(t) \right ) d\lambda(t)=0,
$$
\item[(iii)]
for each $E \in \mathcal{L}$ there exists a vector $x_{E} \in X$ such that 
$$
\lim_{k \to \infty} p \left ( \int_{E} f^{(p)}_{k}(t)d\lambda(t) -x_{E} \right ) =0,
$$
where $\mathcal{L}$ is the family of all Lebesgue measurable subsets of $[0,1]$.
\end{itemize}
The vector $x_{E}$ is said to be the \textit{integral by seminorms} of $f$ over $E$, 
and we set 
$$
x_{E} = \int_{E} f(t) d\lambda(t).
$$
\end{definition}
This notion coincides with the Bochner integral in a Banach space. 
For more information about 
the integral by seminorms and the Bochner integral, 
we refer to \cite{BLOND}, \cite{DIES}, \cite{DUN}, \cite{MIK} and 
\cite{GUS}.

A function $F:[0,1] \to X$ is said to be  
\textit{strongly absolutely continuous ($sAC$)},  
if for each $p \in \mathscr{P}$, we have that $\phi_{p} \circ F$ is $sAC$, 
i.e., given $\varepsilon >0$ there exists $\eta_{\varepsilon p}>0$ such that 
for each finite collection $\{ I_{1}, \dotsc, I_{n} \}$ of 
pairwise nonoverlapping intervals in $\mathcal{I}$, we have 
$$
\sum_{j=1}^{n} \lambda(I_{j}) < \eta_{\varepsilon p} 
 \Rightarrow 
\sum_{j=1}^{n} p(\widetilde{F}(I_{j})) < \varepsilon.
$$

\section{The Main Results}

The main results are Theorem \ref{T2.1} and Theorem \ref{T2.2}. 
Using Pettis's Measurability Theorem, 
Theorem II.1.2 in \cite{DIES}, 
and Severini-Egorov-Theorem, 
Theorem III.6.12 in \cite{DUN}, it can be proved 
the following auxiliary lemma.

\begin{lemma}\label{L2.1}
Let $X$ be a Banach space with the norm $|| \cdot ||$ 
and let $f:[0,1] \to X$ be a function. 
If $f$ is Bochner integrable on $[0,1]$,  
then there exists a sequence $(f_{k})$ of measurable simple functions such that
\begin{equation}\label{eqL21.01}
f_{k}([0,1]) \subset f([0,1]),
\quad
\text{for each}
\quad
k \in \mathbb{N},
\end{equation}
and
\begin{equation}\label{eqL21.02}
\lim_{k \to \infty} f_{k}(t) = f(t),
\quad
\text{at almost all}
\quad
t \in [0,1].
\end{equation}
Moreover, $(f_{k})$ converges to $f$ in $\lambda$-measure, i.e., 
$$
\lim_{k \to \infty} \lambda( \{ t \in [0,1] : ||f_{k}(t)-F(t)|| \geq \eta \} ) = 0,
\quad
\text{for every}
\quad
\eta >0.
$$ 
\end{lemma}

We now present a descriptive characterization of the integral by seminorms 
of a function taking values in a complete locally convex space.

\begin{theorem}\label{T2.1}
Let $X$ be a complete locally convex space and let $F:[0,1] \to X$ be a function. 
If $F$ is $sAC$ and  has the limit average range at almost all $t \in [0,1]$, then $F$ is the primitive of an integrable by seminorms function $f$, i.e.,
\begin{equation}\label{eqT21.0}
\widetilde{F}(I) = \int_{I} f(t) d\lambda(t), 
\quad
\text{for all}
\quad
I \in \mathcal{I}. 
\end{equation}
\end{theorem}
\begin{proof}
By hypothesis, there exists a subset 
$Z \subset [0,1]$ with $\lambda(Z)=0$ such that 
$F$ has the limit average range at all 
$t \in [0,1] \setminus Z$.

$Claim ~1.$ 
For any $p \in \mathscr{P}$, the function $\phi_{p} \circ F$ has the limit average range at all 
$t \in [0,1] \setminus Z$.
To see this, assume that an arbitrary $t \in [0,1] \setminus Z$ 
and $\varepsilon >0$ are given. 
Then, since $F$ has the limit average range at $t$, 
there exists $\delta_{\varepsilon,p}>0$ such that
\begin{equation}\label{eqC1.1}
p-diam(A_{F}(t, \delta_{\varepsilon,p})) 
< 
p-diam(A_{F}(t)) + \varepsilon. 
\end{equation}
It is easy to see that 
$$
x \in \overline{A}^{(p)}_{F}(t,\delta) \Leftrightarrow 
\phi_{p}(x) \in \overline{A}^{(\widetilde{p})}_{\phi_{p} \circ F}(t,\delta)
$$
and
$$
\overline{A}^{(\widetilde{p})}_{\phi_{p} \circ F}(t,\delta)
\subset
\overline{A}^{(\overline{p})}_{\phi_{p} \circ F}(t,\delta),
$$
and since
$$
A_{\phi_{p} \circ F}(t) = 
\bigcap_{\delta>0} 
\overline{A}^{(\overline{p})}_{\phi_{p} \circ F}(t,\delta),
$$
it follows that
\begin{equation*}
\begin{split}
\phi_{p} \left ( A_{F}(t) \right ) 
&\subset 
\phi_{p} \left ( A_{F}^{(p)}(t) \right ) 
\subset
\bigcap_{\delta >0} \phi_{p}(\overline{A}^{(p)}_{F}(t,\delta)) \\
&=
\bigcap_{\delta >0} \overline{A}^{(\widetilde{p})}_{\phi_{p} \circ F}(t,\delta) 
\subset 
\bigcap_{\delta >0} \overline{A}^{(\overline{p})}_{\phi_{p} \circ F}(t,\delta) \\
&=
A_{\phi_{p} \circ F}(t).
\end{split}
\end{equation*}
Hence
\begin{equation}\label{eqC1.2} 
p-diam(A_{F}(t)) = \widetilde{p}-diam[\phi_{p} \left ( A_{F}(t) \right )]
\leq 
\overline{p}-diam(A_{\phi_{p} \circ F}(t)). 
\end{equation}
The equality 
$$
p-diam(A_{F}(t, \delta)) =
\overline{p}-diam(A_{\phi_{p} \circ F}(t,\delta))
$$ 
together with \eqref{eqC1.1} and \eqref{eqC1.2} yields
\begin{equation*}
\overline{p}-diam(A_{\phi_{p} \circ F}(t, \delta_{\varepsilon,p})) 
< 
\overline{p}-diam(A_{\phi_{p} \circ F}(t)) + \varepsilon.  
\end{equation*}
This means that $\phi_{p} \circ F$ has the limit average range at $t$, and since $t$ is arbitrary, 
$\phi_{p} \circ F$ has the limit average range at all $t \in [0,1] \setminus Z$.

Let us now show that 
$$
A_{F}(t) \neq \emptyset,
\quad
\text{at all}
\quad
t \in [0,1] \setminus Z.
$$ 
Fix an arbitrary $t \in [0,1] \setminus Z$ 
and suppose that
\begin{equation}\label{eqT31.1}
A_{F}(t) = \bigcap_{p \in \mathscr{P}} A_{F}^{(p)}(t) = \emptyset.
\end{equation} 
Then, for any $p \in \mathscr{P}$, by the definition of the limit average range, 
there is a decreasing sequence $(\delta_{n}^{p})_{n}$ of real numbers such that 
for each $n \in \mathbb{N}$, we have
\begin{equation}\label{eqT31.2}
0 < \delta_{n}^{p} < \frac{1}{n}
\quad
\text{and}
\quad
p-diam(A_{F}(t, \delta_{n}^{p})) < p-diam(A_{F}(t)) + \frac{1}{n}=\frac{1}{n}.
\end{equation}
Define a sequence $(\Delta F(t, h_{n}^{p}))_{n}$ by choosing 
a vector $\Delta F(t, h_{n}^{p}) \in A_{F}(t, \delta_{n}^{p})$,  
for each $n \in \mathbb{N}$. 
Note that
\begin{equation*}
p(\Delta F(t, h_{m}^{p}) - \Delta F(t, h_{n}^{p})) < \frac{1}{n},
\end{equation*}
whenever $m>n$. 
This means that $(\Delta (\phi_{p} \circ F)(t, h_{n}^{p}))_{n}$ 
is a Cauchy sequence. 
Therefore, there exists 
$\overline{w}_{p} \in \overline{X}_{p}$ such that
\begin{equation}\label{eqC2.01}
\lim_{n \to \infty} \overline{p}
\left [
\Delta(\phi_{p} \circ F)(t, h_{n}^{p}) - \overline{w}_{p}
\right ]
=0.
\end{equation}

$Claim ~2.$ 
For each $p \in \mathscr{P}$, we have that the equality
\begin{equation}\label{eqC2.1}
A_{\phi_{p} \circ  F}(t) = \{ \overline{w}_{p} \} 
\end{equation} 
holds.

To see this, fix an arbitrary $p \in \mathscr{P}$. 
Since
$$
A_{\phi_{p} \circ  F}(t) 
\subset 
\overline{A}^{(\overline{p})}_{\phi_{p} \circ  F}(t, \delta^{p}_{n}),
\quad
\text{for all}
\quad
n \in \mathbb{N},
$$ 
and 
\begin{equation*}
\begin{split}
\overline{p}-diam(\overline{A}^{(\overline{p})}_{\phi_{p} \circ  F}(t, \delta^{p}_{n})) 
&=
\overline{p}-diam(A_{\phi_{p} \circ  F}(t, \delta^{p}_{n})) \\
&=
\widetilde{p}-diam(A_{\phi_{p} \circ  F}(t, \delta^{p}_{n})) \\
&=
p-diam(A_{F}(t, \delta^{p}_{n})), 
\end{split}
\end{equation*}
from \eqref{eqT31.2} we obtain
\begin{equation}\label{eqC2.2}
\overline{p}-diam(A_{\phi_{p} \circ  F}(t))=0.
\end{equation}
The equality \eqref{eqC2.01} together with 
$$
A_{\phi_{p} \circ F}(t) = 
\bigcap_{n=1}^{\infty} 
\overline{A}^{(\overline{p})}_{\phi_{p} \circ F} ( t, \frac{1}{n} )
$$
yields that $\overline{w}_{p} \in A_{\phi_{p} \circ  F}(t)$ 
and, therefore by \eqref{eqC2.2}, 
it follows that 
$A_{\phi_{p} \circ  F}(t) = \{ \overline{w}_{p} \}$. 
Since $p$ is arbitrary, the last equality holds for all $p \in \mathscr{P}$.

$Claim~3.$ There exists $w \in X$ such that
\begin{equation}\label{eqC3.1}
\phi_{p}(w)=\overline{w}_{p},
\quad
\text{for all}
\quad
p \in \mathscr{P}.
\end{equation}
Assume that $p,q \in \mathscr{P}$ such that $p \leq q$ are given. 
Since
$$
\lim_{n \to \infty} \overline{q}
\left [
\Delta(\phi_{q} \circ F)(t, h_{n}^{q}) - \overline{w}_{q}
\right ]
=0
$$
we obtain
\begin{equation*}
\begin{split} 
&\lim_{n \to \infty} 
\overline{p}
\left [
\Delta(\phi_{p} \circ F)(t, h_{n}^{q})) - 
\overline{g}_{pq}(\overline{w}_{q})
\right ]
\\
=&
\lim_{n \to \infty} 
\overline{p}
\left [
\widetilde{g}_{pq} \left ( \Delta(\phi_{q} \circ F)(t, h_{n}^{q}) \right ) - 
\overline{g}_{pq}(\overline{w}_{q})
\right ]
\\ 
=&
\lim_{n \to \infty} 
\overline{p}
\left[
\overline{g}_{pq}
\left (
\Delta(\phi_{q} \circ F)(t, h_{n}^{q}) - \overline{w}_{q}
\right )
\right ]
=0.
\end{split}
\end{equation*}
This yields that 
$\overline{g}_{pq}(\overline{w}_{q}) \in A_{\phi_{p} \circ F}(t)$ 
and, therefore from \eqref{eqC2.1}, we obtain 
$$
\overline{g}_{pq}(\overline{w}_{q})=\overline{w}_{p}.
$$ 
Hence, by \eqref{eqSCH}, 
there exists $w \in X$ such that \eqref{eqC3.1} holds true.

By \eqref{eqC2.1} and \eqref{eqC3.1}, we get
$$
\bigcap_{\delta>0} 
\overline{A}^{(\widetilde{p})}_{\phi_{p} \circ F}(t,\delta) = \{ \phi_{p}(w) \},
\quad
\text{for all}
\quad
p \in \mathscr{P},
$$
and since  
$\phi_{p}(\overline{A}^{(p)}_{F}(t,\delta)) 
=
\overline{A}^{\widetilde{p}}_{\phi_{p} \circ F}(t,\delta)$, 
we obtain
$$
w \in \bigcap_{p \in \mathscr{P}} A_{F}^{(p)}(t).
$$
But, this contradicts \eqref{eqT31.1}. 
Consequently, $A_{F}(t) \neq \emptyset$, and since $t$ is arbitrary, 
the last result holds for all $t \in [0,1] \setminus Z$. 
Then, we can choose a vector $x_{t} \in A_{F}(t)$, 
for each $t \in [0,1] \setminus Z$, and define the function 
$f:[0,1] \to X$ as follows
\begin{equation}
\quad
\mathnormal{f(t)}=
\left \{ \begin{array}{ll}
x_{t}  &\text{if }t \in [0,1] \setminus Z \\
0  &\text{if } t \in Z 
\end{array} \right..
\end{equation}
We are going to prove that 
$f$ is an integrable by seminorms function satisfying \eqref{eqT21.0}.

Since $\phi_{p} \circ F$ is $sAC$ and has the limit average range at almost all $t \in [0,1]$,  
by Lemma 2.4 in \cite{SOK00}, 
$\phi_{p} \circ F$ is differentiable almost everywhere on $[0,1]$, and since 
$$
(\phi_{p} \circ f)(t) \in A_{\phi_{p} \circ F}(t),
\quad
\text{at almost all}
\quad
t \in [0,1],  
$$ 
by Lemma 2.1 in \cite{SOK00},  
we obtain 
$$
(\phi_{p} \circ F)'(t) = (\phi_{p} \circ f)(t),
\quad 
\text{at almost all}
\quad
t \in [0,1].
$$ 
Hence, by Theorem \ref{T1.1}, 
$\phi_{p} \circ f$ is Bochner integrable with the primitive 
$\phi_{p} \circ \widetilde{F}$, i.e., 
\begin{equation}\label{eqT31.b}
(\phi_{p} \circ \widetilde{F})(I) = 
(B)\int_{I} (\phi_{p} \circ f)(t)d\lambda(t), 
\quad
\text{for all}
\quad
I \in \mathcal{I}.
\end{equation}

Since $F$ is $sAC$, 
there exists a unique countable additive $\lambda$-continuous vector measure 
$\nu : \mathcal{L} \to X$ such that $\widetilde{F}(I) = \nu(I)$, for all 
$I \in \mathcal{I}$.   
Hence, from \eqref{eqT31.b}, we obtain
\begin{equation}\label{T2e.3}  
(\phi_{p} \circ \nu)(E) = (B)\int_{E} (\phi_{p} \circ f)(t) d\lambda(t),
\quad
\text{for all}
\quad
E \in \mathcal{L}.
\end{equation}

Since $\phi_{p} \circ f$ is Bochner integrable, 
by Lemma \ref{L2.1}, 
there exists a sequence $(f_{k}^{(p)})$ of measurable simple functions 
such that 
\begin{equation}\label{eqDOM}
(\phi_{p} \circ f_{k}^{(p)})([0,1]) \subset 
(\phi_{p} \circ f)([0,1]),
\quad 
\text{for all}
\quad
k \in \mathbb{N}, 
\end{equation}
\begin{equation}\label{T2e.4} 
\lim_{k \to \infty} p \left ( f_{k}^{(p)}(t) - f(t) \right ) =
\lim_{k \to \infty}\widetilde{p} 
\left ( (\phi_{p} \circ f_{k}^{(p)})(t) - (\phi_{p} \circ f)(t)
\right ) =0, 
\end{equation}
at almost all $t \in [0,1]$, 
and the sequence $(\phi_{p} \circ f_{k}^{(p)})$ 
converges to $\phi_{p} \circ f$ in $\lambda$-measure.

By Theorem II.1.2 in \cite{DIES}, 
the function $\widetilde{p}( (\phi_{p} \circ  f)(\cdot) )$ is Lebesgue integrable 
and, therefore it is bounded almost everywhere on $[0,1]$. 
Thus, there exists $M>0$ such that 
$$
p(f(t)) = \widetilde{p}( (\phi_{p} \circ  f)(t) ) \leq M,
\quad
\text{at almost all}
\quad
t \in [0,1].
$$ 
The last result together with \eqref{eqDOM} yields 
$$
p[f_{k}^{(p)}(t)] 
= \widetilde{p}[(\phi_{p} \circ f_{k}^{(p)})(t))]
\leq M,
\quad
\text{at almost all}
\quad
t \in [0,1],
$$
for every $k \in \mathbb{N}$. 
Therefore, by Dominated Convergence Theorem, 
Theorem II.2.3 in \cite{DIES}, 
we obtain
\begin{equation}\label{T2e.5}
\lim_{k \to \infty}
\widetilde{p} 
\left ( 
\int_{E} (\phi_{p} \circ f_{k}^{(p)})(t)d\lambda(t) - 
(B)\int_{E} (\phi_{p} \circ f)(t) d\lambda(t)
\right ) = 0, 
\end{equation}
for each $E \in \mathcal{L}$, 
and
\begin{equation}\label{T2e.6}
\lim_{k \to \infty} \int_{[0,1]} 
p \left ( f_{k}^{(p)}(t) -f(t) \right ) d\lambda(t) =0.
\end{equation}  
By virtue of \eqref{T2e.3} and \eqref{T2e.5}, we obtain
\begin{equation*}
\lim_{k \to \infty} p
\left ( 
\int_{E} f^{(p)}_{k}(t)d\lambda(t) -\nu(E) 
\right ) 
= 0.
\end{equation*}
By Definition \ref{D1.1}, 
the last equality together with \eqref{T2e.4} and \eqref{T2e.6} yields 
that $f$ is integrable by seminorms and
$$
\widetilde{F}(I) = \nu(I)=\int_{I} f(t) d\lambda(t), 
\quad
\text{for each}
\quad
I \in \mathcal{I},
$$ 
and this ends the proof. 
\end{proof}

The next theorem presents full descriptive characterizations of 
the integral by seminorms  in a Frechet space.

\begin{theorem}\label{T2.2}
Let $X$ be a Frechet space and let $f,F:[0.1] \to X$ be functions. 
Assume that $(p_{k})_{k \in \mathbb{N}}$ is a countable family of increasing seminorms on $X$ so that the topology of $X$ is generated by 
$(p_{k})_{k \in \mathbb{N}}$.
Then the following are equivalent:
\begin{itemize}
\item[(i)]
$f$ is the integrable by seminorms with the primitive $F$, i.e.,
\begin{equation*}
\widetilde{F}(I) = \int_{I} f(t) d\lambda(t), 
\quad
\text{for all}
\quad
I \in \mathcal{I}, 
\end{equation*}
\item[(ii)]
$F$ is $sAC$, $F$ has the limit average range at almost all $t \in [0,1]$ and
\begin{equation*}
f(t) \in A_{F}(t),
\text{ almost everywhere on }
[0,1].
\end{equation*}
\end{itemize}
\end{theorem}
\begin{proof}
$(i) \Rightarrow (ii)$ 
Assume that $f$ is the integrable by seminorms with the primitive $F$. 
Then, each function 
$\phi_{p_{k}} \circ f$ is Bochner integrable with the primitive 
$\phi_{p_{k}} \circ \widetilde{F}$, i.e., 
\begin{equation*}
(\phi_{p_{k}} \circ \widetilde{F})(I) = (B) \int_{I} f(t) d\lambda(t),
\quad
\text{for all}
\quad
I \in \mathcal{I}.
\end{equation*}
Hence, by Theorem \ref{T1.1}, 
$\phi_{p_{k}} \circ F$ is $sAC$ and there exists $Z_{k} \subset [0,1]$ with $\lambda(Z_{k}) = 0$ such that 
$(\phi_{p_{k}} \circ F)'(t)$ exists and
\begin{equation*}
(\phi_{p_{k}} \circ F)'(t) = (\phi_{p_{k}} \circ f)(t),
\quad
\text{at all}
\quad
t \in [0,1] \setminus Z_{k}.
\end{equation*}
Further, by Lemma 2.1 in \cite{SOK00},   
$\phi_{p_{k}} \circ F$ has the limit average range at $t$ and
\begin{equation}\label{eqT22.1} 
A_{\phi_{p_{k}} \circ F}(t) = \{ (\phi_{p_{k}} \circ f)(t) \},
\quad
\text{for all}
\quad
t \in [0,1] \setminus Z_{k}.
\end{equation}

We now fix a point $t \in [0,1] \setminus Z$, 
where $Z = \cup_{k=1}^{+\infty} Z_{k}$ and $\lambda(Z)=0$. 
We will prove that
\begin{itemize}
\item[(a)]
$F$ has the limit average range at $t$,
\item[(b)]
$f(t) \in A_{F}(t)$.
\end{itemize}

$(a)$ 
Since $\phi_{p_{k}} \circ F$ has the limit average range at $t$, 
given $\varepsilon > 0$ there exists $\delta_{\varepsilon}^{(k)}>0$ such that
\begin{equation}\label{eqT22.2}
\overline{p}_{k}-diam( A_{\phi_{p} \circ  F}(t, \delta_{\varepsilon}^{(k)} ) ) 
< 
\overline{p}_{k}-diam( A_{\phi_{p} \circ  F}(t)) + \varepsilon 
= 
\varepsilon.
\end{equation}
By inclusions 
$
A_{F}(t) 
\subset 
A^{(p_{k})}_{F}(t) 
$
and
$
\phi_{p_{k}} (A^{(p_{k})}_{F}(t)) 
\subset 
A_{\phi_{p_{k}} \circ F}(t) 
$ 
we obtain
$$
p_{k}-diam(A_{F}(t)) = 0. 
$$  
Hence, the equality
\begin{equation*}
\begin{split}
p_{k}-diam(A_{F}(t, \delta) 
= 
\overline{p}_{k}-diam(A_{\phi_{p} \circ  F}(t, \delta)),
\end{split}
\end{equation*}
together with \eqref{eqT22.2} yields
\begin{equation}\label{eqT22.3}
p_{k}-diam( A_{F}(t, \delta_{\varepsilon}^{(k)} ) ) 
< 
\varepsilon 
=
p_{k}-diam( A_{F}(t) ) + \varepsilon. 
\end{equation} 
Thus, $F$ has the limit average range at $t$.

$(b)$
At first, 
by the same manner as in the proof of Theorem \ref{T2.1},
we can prove that
\begin{equation*}
A_{F}(t) \neq \emptyset.
\end{equation*}
Then, we may choose a vector $x_{t} \in A_{F}(t)$. 
By \eqref{eqT22.1}, we obtain
\begin{equation*}
A_{\phi_{p_{k}} \circ F}(t) = \{ (\phi_{p_{k}} \circ f)(t) \},
\text{ for all }
k \in \mathbb{N}.
\end{equation*}
Hence
$$
\phi_{p_{k}}(x_{t}) =  \phi_{p_{k}}( f(t) ),
\text{ for all }
k \in \mathbb{N},
$$
and since $X$ is Hausdorff, we infer $x_{t} =f(t)$.

By virtue of Theorem \ref{T2.1}, we obtain $(ii) \Rightarrow (i)$, 
and this ends the proof.
\end{proof}

Finally, we give an example of a function $F : [0,1] \to X$ such that
\begin{itemize}
\item
$F$ is the primitive of an integrable by seminorms function $f$, 
\item
$F$ has the limit average range at all $t \in (0,1)$,
\item
$F$ is not differentiable at any $t \in (0,1)$.
\end{itemize}

\begin{example}\label{Example2.1}  
Let $X$ be the locally convex space $\mathbb{R}^{\mathcal{J}}$ with respect to the topology of pointwise convergence, where $\mathcal{J}=[0,1]$. 
Thus, the family of seminorms $(p_{j})_{j \in \mathcal{J}}$ 
defined as follows
$$
p_{j}(x) = |x(j)|
\quad
\text{for each}
\quad
j \in \mathcal{J},
~
x \in X,
$$
determines the topology of $X$. 
It easy to see that $X$ is a complete locally convex space 
and each normed space $(\widetilde{X}_{p_{j}}, \widetilde{p}_{j})$ 
is isometrically isomorphic with $(\mathbb{R},|.|)$. 
Hence, for each $j \in \mathcal{J}$, we can consider
$$
\widetilde{X}_{p_{j}} = \overline{X}_{p_{j}} = 
\mathbb{R}
\qquad
\widetilde{p}_{j} = \overline{p}_{j} = p_{j}
$$
and
$$
\phi_{p_{j}}(x) = x(j)
\quad
\text{for all}
\quad
x \in X. 
$$

Let us now define the function $F:[0,1] \to X$ as follows
\begin{equation*}
F(t) = x_{t}
\quad
\mathnormal{x_{t}(\theta)}=
\left \{ \begin{array}{ll}
\theta &\text{if~}~ 0 \leq \theta \leq t \\
t      &\text{if~}~t < \theta \leq 1 
\end{array} \right.
\quad
\text{for all}
\quad
t \in [0,1].
\end{equation*}


$(a)$ The function $F$ is $sAC$. To see this, fix an arbitrary 
$j \in \mathcal{J}$. Note that
\begin{equation*}
(\phi_{p_{j}} \circ F)(t) = x_{t}(j) = 
\left \{ \begin{array}{ll}
t   & \text{if~}~ 0 \leq t \leq j \\
j   & \text{if~}~ j < t \leq 1 
\end{array} \right..
\end{equation*}
Then, for each two points $t',t'' \in [0,1]$, we have
$$
p_{j}(F(t'')-F(t')) \leq |t''-t'|.
$$
This yields that $\varphi_{p_{j}} \circ F$ is sAC, and since $j$ is arbitrary, 
we infer that $F$ is $sAC$.


$(b)$ 
The function $F$ has the limit average range at all $t \in (0,1)$. 
Fix an arbitrary $t \in (0,1)$, $\delta>0$ and $h>0$ such that 
$$
0<h<\delta 
\quad
\text{and}
\quad
t-h,~ t+h \in (0,1).
$$
Note that
\begin{equation*}
\mathnormal{\Delta F(t,h)(\theta)}= \frac{x_{t+h}(\theta)-x_{t}(\theta)}{h} = x^{+}_{t,h}(\theta) =
\left \{ \begin{array}{ll}
0                                        &  \text{if~}~ 0 \leq \theta \leq t \\
\frac{\theta-t}{h}  &  \text{if~}~ t < \theta \leq t+h \\
1                                        &  \text{if~}~ t+h <\theta \leq 1 
\end{array} \right.
\end{equation*}
and
\begin{equation*}
\mathnormal{\Delta F(t,-h)(\theta)} = \frac{x_{t-h}(\theta)-x_{t}(\theta)}{-h} = x^{-}_{t,h}(\theta) =
\left \{ \begin{array}{ll}
0                                        &  \text{if~}~ 0 \leq \theta \leq t-h \\
\frac{\theta-(t-h)}{h}  &  \text{if~}~ t-h < \theta \leq t \\
1                                        &  \text{if~}~ t <\theta \leq 1 
\end{array} \right..
\end{equation*}

For each $j \in \mathcal{J}$, we have
\begin{equation}\label{eqEx1.0}
\lim_{n \to \infty} q_{j} \left ( \Delta F \left ( t,\frac{1}{n^{2}} \right ) - x^{+}_{t} \right ) = 0,
\quad
\lim_{n \to \infty} q_{j}\left ( \Delta F \left ( t,-\frac{1}{n^{2}} \right ) - x^{-}_{t} \right ) = 0,
\end{equation}
where 
\begin{equation*}
\mathnormal{x^{+}_{t}(\theta)}=
\left \{ \begin{array}{ll}
0         &\text{if~}~ 0 \leq \theta \leq t \\
1         &\text{if~}~ t <\theta \leq 1 
\end{array} \right.
\quad
\text{and}
\quad
\mathnormal{x^{-}_{t}(\theta)}=
\left \{ \begin{array}{ll}
0  &\text{if~}~ 0\leq \theta < t \\
1  &\text{if~}~ t \leq \theta \leq 1 
\end{array} \right..
\end{equation*}
Hence
$$
x^{+}_{t}, x^{-}_{t} \in A_{F}^{(q_{j})}(t)=
\bigcap_{n=1}^{\infty} \overline{A}_{F}^{(q_{j})} \left ( t,\frac{1}{n} \right ),
\text{ for all }
j \in \mathcal{J}.
$$ 
Therefore 
$$
\{ x^{+}_{t}, x^{-}_{t} \} \subset A_{F}(t).
$$

We now assume that $x \in A_{F}(t)$. 
Fix an arbitrary $j \in \mathcal{J}$ such that $0 \leq j <t$.  
Since
$$
x \in A_{F}^{(q_{j})}(t)=
\bigcap_{n=1}^{\infty} \overline{A}_{F}^{(q_{j})}\left ( t,\frac{1}{n} \right ),
$$
there exists a sequence $(h_{n})$ of real numbers such that
\begin{equation}\label{eqEx1.1}
0 < |h_{n}| < \frac{1}{n}
\quad
\text{and}
\quad
\lim_{n \to \infty} q_{j} ( \Delta F(t,h_{n}) - x) = 0.
\end{equation}
From the fact that $j < t$, it follows that there exists $n_{j}^{-} \in \mathbb{N}$ such that for each $n \in \mathbb{N}$, we have
\begin{equation*}
n \geq n_{j}^{-} \Rightarrow 
\Delta F(t,h_{n})(j) = 0.
\end{equation*}
The last result together with \eqref{eqEx1.1} yields $x(j)=0$, and since $j$ is arbitrary, $x(j)=0$ for all $0 \leq j <t$. 
Similarly, we get that $x(j)=1$ for all $t < j \leq 1$. 

Let us now consider $j=t$. 
By the same way as above 
there exists a sequence $(h_{n})$ of real numbers such that \eqref{eqEx1.1} is satisfied. There exists a subsequence $(h_{n_{k}})$ of $(h_{n})$ such that 
$h_{n_{k}} >0$ for all $k$, or $h_{n_{k}} < 0$ for all $k$. 
In the first case, we get $x(t) = 0$, and similarly $x(t) = 1$ for the second case. 
Thus, $x \in \{ x^{+}_{t}, x^{-}_{t} \}$, and we infer
$$
A_{F}(t) = \{ x^{+}_{t}, x^{-}_{t} \}. 
$$

Note that if $j \neq t$, then for $0<\delta<|t-j|$ we have 
\begin{equation*}
p_{j}-diam(A_{F}(t,\delta))=0
\quad
\text{and} 
\quad
p_{j}-diam(A_{F}(t))=0,
\end{equation*}
otherwise if $j=t$, then for each $\delta > 0$, we have
\begin{equation*}
p_{t}-diam(A_{F}(t,\delta))=1
\quad
\text{and}
\quad
p_{t}-diam(A_{F}(t))=1.
\end{equation*}
Therefore, $F$ has the limit average range at $t$, and since $t$ is arbitrary we infer that $F$ has the limit average range at all $t \in (0,1)$.

By virtue of Theorem \ref{T2.1},  
the functions $f_{1}, f_{2} : [0,1] \to X$ defined as follows
$$
f_{1}(t) = x^{+}_{t},
\quad 
f_{2}(t) = x^{-}_{t},
\quad
\text{for all}
\quad
t \in [0,1],
$$ 
are integrable by seminorms with the primitive $F$.

$(c)$ The function $F$ is not differentiable at any $t \in (0,1)$. 
Indeed, \eqref{eqEx1.0} yields immediately desired  result.
\end{example}

\end{document}